\documentclass[11pt]{article}

\usepackage{amsmath,amssymb,amsfonts,amsthm}
\usepackage{mathtools}
\usepackage{enumitem}
\usepackage[a4paper,margin=1in]{geometry}
\usepackage{hyperref}

\newtheorem{theorem}{Theorem}[section]
\newtheorem{proposition}[theorem]{Proposition}

\newtheorem{corollary}[theorem]{Corollary}

\theoremstyle{definition}
\newtheorem{definition}[theorem]{Definition}
\newtheorem{example}[theorem]{Example}

\theoremstyle{remark}
\newtheorem*{remark}{Remark}
\newcommand{\C}{\mathcal C}

\newcommand{\A}{\mathcal A}
\newcommand{\B}{\mathcal B}
\newcommand{\T}{\mathcal T}
\newcommand{\F}{\mathcal F}

\newcommand{\LoewyLength}{\operatorname{LoewyLength}}
\newcommand{\Shad}{\operatorname{Shad}}

\title{Universal Semisimple Reconstruction and Obstruction Theory}

\author{
Joe Gildea\\
Department of Computing Science and Mathematics,\\
School of Informatics and Creative Arts,\\
Dundalk Institute of Technology\\
\texttt{gildeajoe@gmail.com}}
\date{}

\begin{document}

\maketitle

\begin{abstract}
We develop a semisimple obstruction theory for finite-length abelian
categories. To each object \(X\) we associate a canonical obstruction
subobject \(\mathcal O(X)\), consisting of the part invisible to all
morphisms into semisimple objects. The quotient
\(X/\mathcal O(X)\) is semisimple and universal among such morphisms,
yielding a canonical semisimplification reflector and characterizing
semisimplicity by the vanishing of obstruction. Iterating the construction gives an obstruction filtration which, in
finite-length module categories, recovers the classical radical
filtration and Loewy theory. Obstruction is also heart-dependent:
under Happel--Reiten--Smal{\o} tilting, extension data can appear as new
obstruction layers, leading to obstruction shadows and visible torsion.
Thus classical radical theory arises as the fixed-heart
module-theoretic instance of a categorical framework for semisimple
reconstruction, persistence, and variation of heart.
\end{abstract}



\section{Introduction}

Semisimplicity occupies a central position throughout algebra,
representation theory, and category theory. In a semisimple setting,
objects are determined by their simple constituents, every short exact
sequence splits, and structural information is visible through
semisimple quotients. Classical radical theory measures the failure of
this behaviour through constructions such as the Jacobson radical,
module radicals, radical filtrations, and Loewy series
\cite{Jacobson1956Structure,Faith1973Algebra,
AndersonFuller1992Rings,Lam2001FirstCourse,
Pierce,CurtisReiner}.

The purpose of this paper is to develop a categorical obstruction theory
for semisimplicity based on semisimple detectability and universal
reconstruction. Given an object \(X\) in a finite-length abelian
category \(\C\), we define the obstruction subobject
\[
\mathcal O(X)
=
\bigcap_{f:X\to S,\;S\in\mathcal C_{ss}}
\ker(f),
\]
which records the part of \(X\) invisible to every morphism into a
semisimple object. The construction is intrinsic to the ambient
category: it requires no ring structure, maximal ideal theory, or
projective presentation.

In the setting of finite-length abelian categories and categorical
algebra \cite{Freyd,MacLane,Serre}, we show that
\(X/\mathcal O(X)\) is semisimple and universal among morphisms from
\(X\) into semisimple objects. Consequently, the assignment
\(X\mapsto X/\mathcal O(X)\) defines a canonical semisimplification
reflector. The theory develops from this construction in three
interconnected directions.

\medskip

\noindent
\textbf{Reconstruction.}
The obstruction subobject determines the universal semisimple quotient
and semisimplification reflector, and an object is semisimple precisely
when its obstruction vanishes. Thus \(\mathcal O(X)\) measures the
failure of \(X\) to be recoverable from semisimple data.

\medskip

\noindent
\textbf{Persistence.}
Iteration gives the canonical filtration
\[
X
\supseteq
\mathcal O(X)
\supseteq
\mathcal O^2(X)
\supseteq
\cdots,
\]
whose depth measures the persistence of non-semisimple structure under
successive semisimple collapse. The filtration is compatible with
extensions and quotients, and in finite-length module categories it
recovers the classical radical filtration and Loewy theory.

\medskip

\noindent
\textbf{Heart dependence and tilting.}
Obstruction depends on the ambient heart. In a triangulated category
with finite-length hearts, semisimple detectability may therefore
change under variation of heart. We study this phenomenon under
Happel--Reiten--Smal{\o} tilting \cite{HRS}, where extension data
crossing a torsion pair can appear as new obstruction layers. This
leads to obstruction shadows and visible torsion, and to rigidity
phenomena in which obstruction data recover torsion-theoretic
information associated with the tilt. In this way the theory connects
semisimple obstruction with the geometry of hearts and the structure
of derived categories
\cite{HappelTriangulated,Neeman,AssemSimsonSkowronski}.

Although the theory recovers classical radical behaviour, it is not
defined in radical-theoretic terms. For finite-length module categories
we prove that \(\mathcal O(M)=\operatorname{rad}(M)\), so the obstruction
filtration recovers the radical filtration, Loewy theory, and the
Jacobson radical of the regular module. These are recovery theorems:
the obstruction invariant itself is defined categorically through
semisimple detectability, independently of rings, ideals, and radicals.
The same framework therefore applies in settings where classical
radical theory is unavailable, including varying hearts of triangulated
categories and broader categorical representation-theoretic contexts
\cite{Benson,EtingofGelaki2015Tensor}. From this perspective, classical
radical theory is the fixed-heart module-theoretic instance of a more
general obstruction theory governing reconstruction, persistence, and
heart-dependent semisimple detectability.

\medskip

\noindent
\textbf{Organization of the paper.}
Section~\ref{sec:reconstruction} develops the obstruction subobject,
universal semisimple quotient, and semisimplification reflector.
Section~\ref{sec:persistence} studies obstruction filtrations,
persistence, nilpotence, and Loewy recovery.
Section~\ref{sec:tilting} investigates heart-dependent obstruction under
HRS tilting, including obstruction shadows, visible torsion, and
rigidity.
Section~\ref{sec:classical} compares the obstruction formalism with
classical radical theory in module categories.
Section~\ref{sec:examples} gives illustrative examples, and
Section~\ref{sec:conclusion} concludes the paper.

\section{Reconstruction by Semisimple Collapse}
\label{sec:reconstruction}

The guiding idea is that semisimple objects provide the basic detectors
of structure in a finite-length abelian category. A subobject of \(X\)
that is annihilated by every morphism from \(X\) into a semisimple
object is therefore invisible to semisimple detection. The intersection
of all such kernels is the largest subobject with this property, and
hence measures the precise obstruction to recovering \(X\) from its
semisimple quotient data.

\subsection{The semisimple obstruction subobject}

The obstruction subobject isolates the part of \(X\) that cannot be
seen by semisimple objects. If semisimple data are to provide a means of
reconstructing \(X\), then the natural strategy is to remove precisely
this invisible component. One is therefore led to consider the quotient
\(X/\mathcal O(X)\). The following result shows that this quotient
enjoys a universal property with respect to morphisms into semisimple
objects and hence represents the maximal semisimple information carried
by \(X\).

\begin{definition}
\label{def:obstruction_subobject}
Let $\C$ be a finite-length abelian category and let $X\in\C$.
Define
\[
\mathcal{O}(X)
=
\bigcap_{f:X\to S,\;S\in\mathcal{C}_{ss}}
\ker(f),
\]
where \(\mathcal C_{\mathrm{ss}}\) denotes the full subcategory of
semisimple objects of \(\mathcal C\).
We call \(\mathcal{O}(X)\) the \emph{obstruction subobject} of \(X\).
\end{definition}

Equivalently, \(\mathcal O(X)\) is the largest subobject of \(X\)
that is invisible to all semisimple quotients.

\begin{remark}
The construction depends only on morphisms from \(X\) into semisimple
objects. It is therefore intrinsic to the finite-length abelian category
\(\C\), rather than to any chosen ring presentation. Thus
\(\mathcal O(X)\) measures the part of \(X\) that cannot be detected by
semisimple objects and may be viewed as an obstruction to semisimple
reconstruction.
\end{remark}

\subsection{Universal semisimple quotient}

The obstruction subobject \(\mathcal O(X)\) captures precisely the part
of \(X\) that is invisible to semisimple objects. It is therefore
natural to ask whether quotienting by \(\mathcal O(X)\) produces a
canonical object that retains all semisimple information carried by
\(X\). The following result shows that this is indeed the case: the
quotient \(X/\mathcal O(X)\) is universal among morphisms from \(X\)
into semisimple objects and hence provides a canonical semisimple
reconstruction of \(X\).

The next result shows that this obstruction is not merely a distinguished
subobject. It determines a canonical quotient through which every
morphism from \(X\) into a semisimple object factors. Thus semisimple
obstruction gives rise to a universal reconstruction procedure.

\begin{theorem}\label{thm:universal_semisimple_quotient}
For every \(X\in\C\), the quotient $X/\mathcal O(X)$ is semisimple and universal among morphisms from \(X\) into semisimple
objects.
\end{theorem}

\begin{proof}
Let $R=\mathcal O(X)$ and let $q:X\to X/R$ be the quotient morphism. By definition,
\[
R=\bigcap_{f:X\to S,\;S\in\mathcal C_{ss}}\ker(f).
\]
Hence every morphism \(f:X\to S\) into a semisimple object satisfies
\(R\subseteq\ker(f)\). Therefore \(f\) factors uniquely through \(q\).
Thus every semisimple observation of \(X\) is determined entirely by the
quotient \(X/R\), establishing the universal factorization property. 
It remains to show that \(X/R\) is semisimple. Let $Q=X/R$. Since \(R\) is the intersection of the kernels of all morphisms from
\(X\) into semisimple objects, the intersection of the kernels of all
morphisms $Q\to S$, $S\in\mathcal C_{ss}$, is zero. Because \(Q\) has finite length, there exist finitely many morphisms
$g_i:Q\to S_i$, $i=1,\dots,n$, with each \(S_i\) semisimple, whose kernels already have trivial
intersection. Consequently,
\[
g=(g_1,\dots,g_n):
Q\longrightarrow
\prod_{i=1}^{n}S_i
\]
is a monomorphism. The finite product \(\prod_{i=1}^{n}S_i\) is semisimple, and every
subobject of a semisimple object in an abelian category is semisimple.
Hence \(Q\) is semisimple. Therefore \(X/\mathcal O(X)\) is semisimple and universal among
morphisms from \(X\) into semisimple objects.
\end{proof}

The previous theorem identifies \(X/\mathcal O(X)\) as the canonical
semisimple image of \(X\). The next result shows that the obstruction
subobject is characterized entirely by this universal property. In this
sense, \(\mathcal O(X)\) is exactly the part of \(X\) that prevents
semisimple reconstruction.

\begin{theorem}\label{thm:obstruction_characterization}
Let \(X\in\C\). The subobject \(\mathcal O(X)\subseteq X\) is the unique
subobject \(R_X\subseteq X\) satisfying the following conditions:
\begin{enumerate}
\item \(R_X\) is the largest subobject annihilated by every morphism
from \(X\) into a semisimple object;

\item \(X/R_X\) is semisimple and universal among morphisms from \(X\)
into semisimple objects;

\item for every subobject \(N\subseteq X\), \(X/N\) is semisimple if and only if \(R_X\subseteq N\).
\end{enumerate}
\end{theorem}
\begin{proof}
Let
\[
R_X=\mathcal O(X)
=
\bigcap_{f:X\to S,\;S\in\mathcal C_{ss}}\ker(f).
\]

First, \(R_X\) is annihilated by every morphism from \(X\) into a semisimple object by definition. If \(N\subseteq X\) is any subobject
annihilated by every such morphism, then \(N\subseteq \ker(f)\) for every morphism \(f:X\to S\) with \(S\in\mathcal C_{ss}\). Hence
\[
N\subseteq
\bigcap_{f:X\to S,\;S\in\mathcal C_{ss}}\ker(f)
=
R_X.
\]
Thus \(R_X\) is the largest subobject invisible to all semisimple
morphisms. This proves condition (1). By Theorem~\ref{thm:universal_semisimple_quotient}, the quotient
\(X/R_X\) is semisimple and every morphism from \(X\) into a semisimple object factors uniquely through the quotient map
\(X\to X/R_X\). Thus \(R_X\) satisfies condition (2). Now let \(N\subseteq X\). If \(X/N\) is semisimple, then the quotient map
\(X\to X/N\) is a morphism from \(X\) into a semisimple object. Hence, by the
definition of \(R_X\), \(R_X\subseteq N\). Conversely, if \(R_X\subseteq N\), then \(X/N\) is a quotient of the
semisimple object \(X/R_X\). Since quotients of semisimple objects in an abelian category are semisimple, \(X/N\) is
semisimple. Therefore \(X/N\) is semisimple if and only if \(R_X\subseteq N\). This proves condition (3).

Finally, condition (3) uniquely determines \(R_X\). Indeed, suppose that
\(R'_X\subseteq X\) also satisfies condition (3). Since \(X/R_X\) is
semisimple, condition (3) for \(R'_X\) gives \(R'_X\subseteq R_X\).
Similarly, since \(X/R'_X\) is semisimple, condition (3) for \(R_X\)
gives \(R_X\subseteq R'_X\). Hence \(R_X=R'_X\). Equivalently, \(R_X\)
is precisely the smallest subobject whose quotient is semisimple, and
is therefore the minimal obstruction to obtaining a semisimple quotient
of \(X\). Therefore \(\mathcal O(X)\) is the unique subobject satisfying
the stated conditions.
\end{proof}

\begin{remark}
Together, Theorems~\ref{thm:universal_semisimple_quotient}
and~\ref{thm:obstruction_characterization} show that semisimple
reconstruction is controlled entirely by the obstruction subobject
\(\mathcal O(X)\). The obstruction records precisely the part of \(X\)
that cannot be detected by semisimple objects, while the quotient
\(X/\mathcal O(X)\) captures all semisimple information carried by \(X\).
\end{remark}

\subsection{Functoriality and semisimplification}

The universal characterization identifies the semisimple obstruction as
a canonical subobject. The next result shows that this construction is
compatible with morphisms and therefore extends from objects to a
functorial semisimplification process.

Functoriality shows that semisimple reconstruction is compatible with
morphisms. The next result identifies this construction categorically as
the reflector onto the full subcategory of semisimple objects.

\begin{theorem}
\label{thm:functoriality}
Every morphism \(f:X\to Y\) induces a unique morphism
\[
\overline f:
X/\mathcal O(X)\longrightarrow Y/\mathcal O(Y)
\]
such that $q_Y\circ f=\overline f\circ q_X$, where $q_X:X\to X/\mathcal O(X)$ and $q_Y:Y\to Y/\mathcal O(Y)$ are the canonical quotient maps.
\end{theorem}

\begin{proof}
Let $g:Y\to S$ be any morphism into a semisimple object. Then $g\circ f:X\to S$ is also a morphism into a semisimple object. Hence
$\mathcal O(X)\subseteq \ker(g\circ f)$. Therefore $f\bigl(\mathcal O(X)\bigr)\subseteq \ker(g)$. Since this holds for every morphism $g:Y\to S$
with \(S\) semisimple, we obtain
\[
f\bigl(\mathcal O(X)\bigr)
\subseteq
\bigcap_{g:Y\to S}\ker(g)
=
\mathcal O(Y).
\]

Thus \(f\) maps the obstruction subobject of \(X\) into the
obstruction subobject of \(Y\). Consequently, \(f\) descends uniquely
to a morphism $\overline f:X/\mathcal O(X)\to Y/\mathcal O(Y)$ satisfying $q_Y\circ f=\overline f\circ q_X$.
\end{proof}

Functoriality allows the universal semisimple quotient to be assembled
into a global construction on the category. The resulting functor is not
merely a quotient operation: it is the reflector onto the full
subcategory of semisimple objects.

\begin{theorem}
\label{thm:semisimplification_reflector}
The assignment $X\longmapsto X/\mathcal O(X)$ defines a functor $\mathsf S:\C\to\mathcal C_{ss}$ left adjoint to the inclusion $\iota:\mathcal C_{ss}\hookrightarrow \C$.
\end{theorem}

\begin{proof}
For each \(X\in\mathcal C\), define $\mathsf S(X)=X/\mathcal O(X)$. By Theorem~\ref{thm:universal_semisimple_quotient}, \(\mathsf S(X)\) is semisimple. Hence \(\mathsf S\) is
objectwise a map from \(\mathcal C\) to \(\mathcal C_{ss}\). Let $f:X\to Y$ be a morphism in \(\mathcal C\). By Theorem~\ref{thm:functoriality},
$f\bigl(\mathcal O(X)\bigr) \subseteq \mathcal O(Y)$. Therefore \(f\) descends uniquely to a morphism $\mathsf S(f): X/\mathcal O(X) \to Y/\mathcal O(Y)$.
Thus \(\mathsf S\) is a functor. Now let \(S\in\mathcal C_{ss}\). By the universal semisimple quotient property, every morphism
$f:X\to S$ factors uniquely through the quotient morphism $q_X:X\to \mathsf S(X)$. That is, there exists a unique morphism $\overline f:\mathsf S(X)\to S$
such that $f=\overline f\circ q_X$. In other words, every morphism from \(X\) into a semisimple object is
completely determined by the semisimple reconstruction
\(\mathsf S(X)\). Consequently, composition with \(q_X\) induces a natural bijection
\[
\operatorname{Hom}_{\mathcal C_{ss}}(\mathsf S(X),S)
\cong
\operatorname{Hom}_{\mathcal C}(X,\iota S).
\]

Therefore $\mathsf S\dashv \iota$. Hence semisimple reconstruction defines a reflector from \(\mathcal C\) onto its full subcategory of semisimple objects. In particular, \(\mathsf S(X)\) is the canonical semisimple approximation to \(X\).
\end{proof}

The reflector property shows that semisimplification is universal among
morphisms into semisimple objects. The following result strengthens this
observation by showing that the semisimplification functor is uniquely
determined by this universal property.

\begin{theorem}
\label{thm:uniqueness_semisimplification}
Let $\mathsf T:\C\to\mathcal C_{ss}$ be a functor equipped with natural morphisms $\tau_X:X\to\mathsf T(X)$ such that every morphism from \(X\) into a semisimple object factors
uniquely through \(\tau_X\). Then $\mathsf S\cong\mathsf T$ naturally.
\end{theorem}

\begin{proof}
Let $q_X:X\to \mathsf S(X)$ denote the universal semisimple quotient map. Since \(\mathsf T(X)\) is semisimple and \(\tau_X:X\to \mathsf T(X)\) is a morphism from \(X\) into a
semisimple object, the universal property of \(q_X\) yields a unique morphism $a_X:\mathsf S(X)\to \mathsf T(X)$ such that $\tau_X=a_X\circ q_X$. Conversely, since \(\mathsf S(X)\) is semisimple, the assumed universal property of \(\tau_X\) yields a unique morphism $b_X:\mathsf T(X)\to \mathsf S(X)$ such that $q_X=b_X\circ \tau_X$. We claim that \(a_X\) and \(b_X\) are inverse isomorphisms. Indeed, $(b_X\circ a_X)\circ q_X = b_X\circ \tau_X = q_X$. Both \(b_X\circ a_X\) and \(\mathrm{id}_{\mathsf S(X)}\) are morphisms $\mathsf S(X)\to \mathsf S(X)$
whose composites with \(q_X\) equal \(q_X\). Since \(q_X\) is universal for morphisms from \(X\) into semisimple objects, such a factorization is unique. Hence
$b_X\circ a_X = \mathrm{id}_{\mathsf S(X)}$. Similarly, $(a_X\circ b_X)\circ \tau_X=a_X\circ q_X=\tau_X$. By the assumed uniqueness property of \(\tau_X\), we obtain
$a_X\circ b_X = \mathrm{id}_{\mathsf T(X)}$. Thus $a_X:\mathsf S(X)\to \mathsf T(X)$ is an isomorphism for every \(X\). It remains to verify naturality. Let
$f:X\to Y$ be a morphism in \(\mathcal C\). We must show that $a_Y\circ \mathsf S(f)=\mathsf T(f)\circ a_X$. Both sides are morphisms
$\mathsf S(X)\to \mathsf T(Y)$. After precomposition with \(q_X\), we obtain $(a_Y\circ \mathsf S(f))\circ q_X=a_Y\circ q_Y\circ f=\tau_Y\circ f$, where naturality of \(q\) gives
$\mathsf S(f)\circ q_X= q_Y\circ f$. On the other hand, $(\mathsf T(f)\circ a_X)\circ q_X =\mathsf T(f)\circ \tau_X=\tau_Y\circ f$, by naturality of \(\tau\). Thus both composites have the same precomposition with \(q_X\). By the universal property of \(q_X\), they are equal. Hence the isomorphisms $a_X:\mathsf S(X)\to \mathsf T(X)$ assemble into a natural isomorphism
$\mathsf S\cong \mathsf T$. Therefore the semisimplification functor \(\mathsf S\) is uniquely determined, up to natural isomorphism, by its universal factorization property.
\end{proof}

\subsection{Semisimple objects as the zero-obstruction regime}

The preceding results identify \(\mathcal O(X)\) as the canonical
obstruction to semisimple reconstruction. A natural question is
therefore whether the vanishing of this obstruction characterizes
semisimplicity itself. The next result shows that this is indeed the
case: an object has zero obstruction precisely when it is semisimple.
Thus semisimple objects constitute the zero-obstruction regime of the
theory.

\begin{theorem}
\label{thm:vanishing_criterion}
For \(X\in\C\), the following are equivalent:
\begin{enumerate}
\item \(X\) is semisimple;

\item \(\mathcal O(X)=0\);

\item the quotient map $X\to X/\mathcal O(X)$ is an isomorphism;

\item every nonzero subobject of \(X\) is detected by some morphism from
\(X\) into a semisimple object.
\end{enumerate}
\end{theorem}

\begin{proof}
We prove the equivalence of the four conditions. Suppose first that \(X\) is semisimple. Then the identity morphism $\mathrm{id}_X:X\to X$ is a morphism from \(X\) into a semisimple object. Since
\[
\mathcal O(X)
=
\bigcap_{f:X\to S,\;S\in\mathcal C_{ss}}\ker(f),
\]
and \(\ker(\mathrm{id}_X)=0\), it follows that $\mathcal O(X)=0$. Thus \textup{(1)} implies \textup{(2)}. Now suppose that $\mathcal O(X)=0$. The universal semisimple quotient map
$q_X:X\to X/\mathcal O(X)$ is an epimorphism, and its kernel is \(\mathcal O(X)=0\), so it is also a monomorphism. Since \(\mathcal C\) is abelian, every morphism that is both monic and epic is an isomorphism. Hence \(q_X\) is an isomorphism. Thus \textup{(2)} implies \textup{(3)}. If the quotient map $X\to X/\mathcal O(X)$ is an isomorphism, then \(X\) is isomorphic to $X/\mathcal O(X)$. By Theorem~\ref{thm:universal_semisimple_quotient}, this quotient is semisimple. Hence \(X\) is semisimple. Thus \textup{(3)} implies \textup{(1)}. It remains to relate these conditions to \textup{(4)}. Suppose first that $\mathcal O(X)=0$. Let \(N\subseteq X\) be a nonzero subobject. If every morphism $f:X\to S$ with \(S\in\mathcal C_{ss}\) satisfied $N\subseteq \ker(f)$, then
\[
N\subseteq
\bigcap_{f:X\to S,\;S\in\mathcal C_{ss}}\ker(f)
=
\mathcal O(X)
=
0,
\]
contradicting \(N\neq 0\). Therefore there exists a morphism $f:X\to S$ into a semisimple object such that $N\not\subseteq \ker(f)$. Hence \textup{(2)} implies \textup{(4)}. Conversely, suppose \textup{(4)} holds. If $\mathcal O(X)\neq 0$, then taking $N=\mathcal O(X)$ gives a nonzero subobject of \(X\). By definition, $\mathcal O(X)\subseteq \ker(f)$
for every morphism $f:X\to S$ with \(S\in\mathcal C_{ss}\). This contradicts \textup{(4)}. Therefore $\mathcal O(X)=0$. Thus \textup{(4)} implies \textup{(2)}. Combining the implications proves that all four conditions are equivalent.
\end{proof}

The preceding criterion identifies semisimplicity at the level of
individual objects. The next result shows that the same phenomenon
extends to the category as a whole: the semisimplification functor is
exact precisely when no nontrivial semisimple obstruction exists
anywhere in the category.

The semisimple reconstruction criterion identifies semisimple objects
as precisely those with vanishing obstruction. The next result shows
that this phenomenon extends from individual objects to the ambient
category itself. In particular, exactness of the semisimplification
functor is equivalent to the absence of nontrivial obstruction
throughout the category.

\begin{theorem}
\label{thm:global_vanishing_criterion}
The semisimplification reflector $\mathsf S:\C\to\mathcal C_{ss}$ is exact if and only if \(\C\) is semisimple.
\end{theorem}

\begin{proof}
If \(\mathcal C\) is semisimple, then every object is semisimple. Hence $\mathcal O(X)=0$ for every \(X\in\mathcal C\), and therefore $\mathsf S(X)=X$. Thus \(\mathsf S\) is naturally isomorphic to the identity functor and is therefore exact. Conversely, suppose that \(\mathsf S\) is exact. For any object
\(X\in\mathcal C\), consider the canonical short exact sequence
\[
0
\to
\mathcal O(X)
\to
X
\xrightarrow{q_X}
\mathsf S(X)
\to
0.
\]

Applying the exact functor \(\mathsf S\) yields an exact sequence
\[
0
\to
\mathsf S(\mathcal O(X))
\to
\mathsf S(X)
\xrightarrow{\mathsf S(q_X)}
\mathsf S(\mathsf S(X))
\to
0.
\]

Since \(\mathsf S(X)\) is semisimple, it is fixed by the reflector: $\mathsf S(\mathsf S(X)) \cong \mathsf S(X)$. Moreover, \(\mathsf S(q_X)\) is an isomorphism. Hence its kernel is zero, and exactness implies $\mathsf S(\mathcal O(X))=0$. Now let \(Y\) be a nonzero object of \(\mathcal C\). Since \(\mathcal C\) has finite length, \(Y\) admits a simple quotient. Consequently there
exists a nonzero morphism $Y\to S$ with \(S\) simple, hence semisimple. By the universal semisimple quotient property, $\mathsf S(Y)\neq 0$. Applying this observation to
$Y=\mathcal O(X)$, the equality $\mathsf S(\mathcal O(X))=0$ forces $\mathcal O(X)=0$. Thus $\mathcal O(X)=0$ for every \(X\in\mathcal C\). By Theorem~\ref{thm:vanishing_criterion}, every object of \(\mathcal C\) is semisimple. Hence \(\mathcal C\) is semisimple.
\end{proof}

\section{Persistence and Obstruction Filtrations}
\label{sec:persistence}

The reconstruction theory developed in the previous section identifies
\(\mathcal O(X)\) as the canonical obstruction to semisimple
reconstruction. However, removing this obstruction need not eliminate
all non-semisimple structure at once. It is therefore natural to ask
how obstruction persists under repeated semisimplification. Iterating
the obstruction construction gives rise to a canonical filtration that
measures successive layers of obstruction and records the depth of
non-semisimple structure within an object.

\subsection{The obstruction filtration}

The basic invariant governing this persistence phenomenon is the
following filtration.

\begin{definition}
Define $\mathcal O^{0}(X)=X$, and for \(n\geq 0\), $\mathcal O^{n+1}(X) =\mathcal O\!\left(\mathcal O^{n}(X)\right)$. The resulting descending filtration
$X
\supseteq
\mathcal O(X)
\supseteq
\mathcal O^{2}(X)
\supseteq
\cdots$ is called the \emph{obstruction filtration} of \(X\).
\end{definition}

The obstruction filtration records successive layers of semisimple
obstruction. Its length measures how many iterations of the obstruction
construction are required before all obstruction disappears.

\begin{definition}
The \emph{obstruction depth} of \(X\) is $\operatorname{depth}_{\mathcal O}(X)=\min\{d\geq 0:\mathcal O^{d}(X)=0\}$, when such a \(d\) exists.
\end{definition}

\subsection{Nilpotence of obstruction}

The obstruction filtration measures how semisimple obstruction persists
under repeated semisimplification. For objects of finite length, this
process cannot continue indefinitely. The next result shows that the
obstruction filtration is necessarily finite and therefore defines a
numerical measure of obstruction complexity.

\begin{theorem}
\label{thm:obstruction_nilpotence}
If \(\C\) is finite-length, then for every \(X\in\C\) the obstruction
filtration terminates after finitely many steps.
\end{theorem}

\begin{proof}
Since \(\C\) is finite-length, every descending chain of subobjects of
\(X\) stabilizes. Therefore the obstruction filtration
$X \supseteq \mathcal O(X) \supseteq \mathcal O^{2}(X) \supseteq \cdots$
eventually stabilizes. Thus there exists \(N\) such that
$\mathcal O^{N}(X)=\mathcal O^{N+1}(X)$. Let
$Y=\mathcal O^{N}(X)$. Then $\mathcal O(Y)=Y$. We claim that \(Y=0\).
Suppose instead that \(Y\neq 0\). Since \(Y\) has finite length, it
admits a simple quotient $Y\to L$. Because \(L\) is simple, it is
semisimple. However, if $f:Y\to S$ is any morphism into a semisimple
object, then by definition of the obstruction subobject,
$\mathcal O(Y)\subseteq\ker(f)$. Since $\mathcal O(Y)=Y$, it follows
that $Y\subseteq\ker(f)$, and hence $\ker(f)=Y$. Therefore \(f=0\).
Thus every morphism from \(Y\) into a semisimple object is zero. This
contradicts the existence of the nonzero quotient map $Y\to L$. Hence
\(Y=0\). Therefore $\mathcal O^{N}(X)=0$, and the obstruction filtration
terminates after finitely many steps.
\end{proof}

The nilpotence theorem shows that obstruction can persist only for
finitely many stages. Consequently, every object admits a finite
obstruction depth, and the obstruction filtration provides a finite
measure of the complexity of semisimple reconstruction.

\subsection{Persistence under extensions}

The obstruction filtration provides a quantitative measure of how far an
object lies from the semisimple regime. Having established its
relationship with classical Loewy theory, we now investigate how
obstruction behaves under the fundamental constructions of homological
algebra. The next results show that obstruction depth exhibits a
persistence phenomenon under extensions and quotients, reflecting the
way semisimple obstruction propagates through exact sequences.

The first persistence theorem shows that obstruction depth is controlled
under extensions. In particular, the obstruction complexity of an
extension cannot exceed that of its constituent pieces by an arbitrary
amount. Thus obstruction depth behaves as a robust measure of
non-semisimple structure across short exact sequences.

\begin{proposition}
\label{prop:extension_persistence}
Let $0\to A\to E\to B\to 0$ be a short exact sequence in \(\C\). Then the obstruction depth of \(E\)
is bounded in terms of the obstruction depths of \(A\) and \(B\). More
precisely,
\[
\operatorname{depth}_{\mathcal O}(E)
\leq
\operatorname{depth}_{\mathcal O}(A)
+
\operatorname{depth}_{\mathcal O}(B).
\]
\end{proposition}

\begin{proof}
Let $p:E\to B$ be the quotient morphism. The guiding idea is that obstruction surviving in \(E\) must either survive in the quotient \(B\) or be confined to the subobject \(A\). We first observe that the obstruction filtration is functorial with respect to morphisms. Indeed, let $f:X\to Y$ be any morphism, and let $q_Y:Y\to Y/\mathcal O(Y)$ be the universal semisimple quotient. Since $Y/\mathcal O(Y)$ is semisimple, the composite $q_Y\circ f:X\to Y/\mathcal O(Y)$ has $\mathcal O(X)$ contained in its kernel. Hence $q_Y\circ f$ vanishes on $\mathcal O(X)$, so $f(\mathcal O(X))\subseteq\ker(q_Y)=\mathcal O(Y)$. Iterating this argument gives $f(\mathcal O^{n}(X))\subseteq\mathcal O^{n}(Y)$ for every \(n\geq 0\).

Applying this to $p:E\to B$, for every \(n\geq 0\), $p\bigl(\mathcal O^{n}(E)\bigr)\subseteq\mathcal O^{n}(B)$. Let $b=\operatorname{depth}_{\mathcal O}(B)$. Then $\mathcal O^{b}(B)=0$. Hence $p\bigl(\mathcal O^{b}(E)\bigr)=0$, and therefore $\mathcal O^{b}(E)\subseteq\ker(p)=A$. Now let $a=\operatorname{depth}_{\mathcal O}(A)$. Then $\mathcal O^{a}(A)=0$. Applying the functoriality of the obstruction filtration to the inclusion $\mathcal O^{b}(E)\subseteq A$ gives $\mathcal O^{a}\bigl(\mathcal O^{b}(E)\bigr)\subseteq\mathcal O^{a}(A)=0$. By definition of the iterated obstruction filtration, $\mathcal O^{a}\bigl(\mathcal O^{b}(E)\bigr)=\mathcal O^{a+b}(E)$. Therefore $\mathcal O^{a+b}(E)=0$. Thus all obstruction in \(E\) is eliminated after at most \(a+b\) iterations of the obstruction construction. Consequently,
\[
\operatorname{depth}_{\mathcal O}(E)
\leq
a+b
=
\operatorname{depth}_{\mathcal O}(A)
+
\operatorname{depth}_{\mathcal O}(B).
\]
\end{proof}

The extension estimate controls obstruction depth globally. The next
result shows that individual layers of the obstruction filtration are
also preserved under passage to quotients. Thus obstruction persists
functorially not only at the level of depth, but throughout the entire
filtration.

\begin{proposition}
\label{prop:quotient_persistence}
Let $q:X\to Y$ be an epimorphism. Then $q\bigl(\mathcal O^{n}(X)\bigr) \subseteq \mathcal O^{n}(Y)$ for all \(n\ge 0\).
\end{proposition}

\begin{proof}
The statement expresses the compatibility of the obstruction filtration with epimorphisms. We proceed by induction on \(n\). For \(n=0\), we have
$\mathcal O^{0}(X)=X$ and $\mathcal O^{0}(Y)=Y$. Since \(q:X\to Y\) is an epimorphism, $q\bigl(\mathcal O^{0}(X)\bigr)=q(X)=Y =\mathcal O^{0}(Y)$,
so the claim holds. Now assume that $q\bigl(\mathcal O^{n}(X)\bigr) \subseteq \mathcal O^{n}(Y)$. Restricting \(q\) yields a morphism
$q_n: \mathcal O^{n}(X) \longrightarrow \mathcal O^{n}(Y)$. By functoriality of the obstruction construction,
$q_n\bigl( \mathcal O(\mathcal O^{n}(X)) \bigr) \subseteq \mathcal O(\mathcal O^{n}(Y))$. Equivalently,
$q\bigl(\mathcal O^{n+1}(X)\bigr) \subseteq \mathcal O^{n+1}(Y)$. Thus the claim holds for \(n+1\), and the result follows by induction.
\end{proof}

Together, the extension and quotient persistence theorems show that
semisimple obstruction is compatible with the basic exact operations of
finite-length abelian categories. Consequently, obstruction depth is not
merely a numerical invariant but part of a functorial filtration theory.
The obstruction filtration therefore provides a systematic framework for
tracking how non-semisimple structure propagates through extensions,
subobjects, and quotients.

\subsection{Persistence Structure}

The preceding results have introduced several manifestations of
semisimple obstruction: the obstruction subobject, the universal
semisimple quotient, the semisimplification reflector, and the
obstruction filtration. The next theorem shows that these are not
independent constructions. Rather, each encodes the same underlying
obstruction-theoretic information and canonically determines the
others. Thus semisimple obstruction theory may be formulated
equivalently from any of these perspectives.

\begin{theorem}
\label{thm:persistence_structure}
The obstruction assignment, universal semisimple quotient maps,
semisimplification reflector, and obstruction filtration canonically
determine one another, up to the unique canonical isomorphisms arising
from their universal properties.
\end{theorem}
\begin{proof}
The obstruction assignment determines the universal semisimple quotient maps by assigning to each object \(X\) the quotient morphism
$q_X: X\to X/\mathcal O(X)$. By Theorem~\ref{thm:universal_semisimple_quotient}, the quotient $X/\mathcal O(X)$ is semisimple and universal among morphisms from \(X\) into semisimple
objects. Conversely, the universal semisimple quotient maps determine the obstruction subobjects as their kernels: $\mathcal O(X) = \ker(q_X)$. Thus the obstruction assignment and the universal semisimple quotient data determine one another. The same quotient maps define the semisimplification functor $\mathsf S(X) = X/\mathcal O(X)$. By Theorem~\ref{thm:functoriality}, every morphism
$f:X\to Y$ induces a morphism $\mathsf S(f): \mathsf S(X)\to\mathsf S(Y)$. The identity and composition laws follow from the uniqueness of the induced morphisms, so the quotient construction determines a functor $\mathsf S:\C\to\mathcal C_{ss}$. By Theorem~\ref{thm:semisimplification_reflector}, this functor is left adjoint to the inclusion $\iota: \mathcal C_{ss}\hookrightarrow\C$. Conversely, the reflector determines the universal quotient maps through its unit morphisms $\eta_X: X\to\mathsf S(X)$, and their kernels recover $\mathcal O(X)$. Hence the reflector and the obstruction assignment determine one another, up to the unique canonical isomorphisms arising from their universal properties. Finally, the obstruction assignment determines the obstruction filtration by iteration: $\mathcal O^{0}(X)=X$, $\mathcal O^{n+1}(X)=\mathcal O\!\left(\mathcal O^{n}(X)\right)$. Conversely, the first nontrivial term of the filtration is precisely $\mathcal O(X)$, so the filtration determines the obstruction assignment. Therefore the obstruction assignment, the universal semisimple quotient
maps, the semisimplification reflector, and the obstruction filtration canonically determine one another, up to the unique canonical isomorphisms arising from their universal properties.
\end{proof}

The Persistence Structure Theorem shows that semisimple obstruction
theory is governed by a single underlying mechanism. The obstruction
subobject records the part of an object that remains invisible to
semisimple detection, the universal semisimple quotient records the
corresponding reconstruction, the semisimplification reflector
globalizes this reconstruction functorially, and the obstruction
filtration measures the persistence of obstruction through successive
stages of semisimplification. Consequently, these constructions are not
independent: each canonically determines the others, and any one of them
encodes the full obstruction-theoretic content of the theory.

\section{Tilting and Heart-Dependent Obstruction}
\label{sec:tilting}

This section extends obstruction theory from individual finite-length
abelian categories to triangulated settings admitting multiple
finite-length hearts. The central question is how obstruction changes
when the ambient heart is varied. We focus on
Happel--Reiten--Smal{\o} tilting and show that tilting can create new
obstruction layers, that these layers encode information about the
underlying torsion pair, and that in favourable situations obstruction
data recover the tilting structure itself. This leads naturally to the
notion of a derived obstruction profile, which records obstruction
behaviour across all finite-length hearts of a triangulated category.

\subsection{Heart-dependent obstruction}

The preceding sections developed obstruction theory within a fixed
finite-length abelian category. In a triangulated setting, however,
obstruction depends on the choice of heart. The purpose of this section
is to investigate how obstruction behaves under variation of heart,
with particular emphasis on Happel--Reiten--Smal{\o} tilting. We shall
see that tilting can create new obstruction layers, that these layers
encode information about the underlying torsion pair, and that in
favourable situations the resulting obstruction data determine the
tilt itself.

The first observation is that obstruction is not an invariant of the
ambient triangulated category alone. Rather, it depends on the heart in
which an object is viewed, since semisimple detectability itself is a
heart-dependent notion.

Let \(D\) be a triangulated category admitting finite-length hearts,
and let \(\A\subset D\) be such a heart.

\begin{definition}
Let \(\A\subset D\) be a finite-length heart and let \(X\in\A\). Define
\[
\mathcal O^{\A}(X)
=
\bigcap_{f:X\to S,\;S\in\A_{ss}}
\ker_{\A}(f),
\]
where \(\A_{ss}\) denotes the full subcategory of semisimple objects
of \(\A\), and where the morphisms, kernels, and intersection are taken
in the abelian category \(\A\). We call \(\mathcal O^{\A}(X)\) the
\emph{obstruction subobject of \(X\) with respect to the heart \(\A\)}.
\end{definition}

\begin{remark}
The obstruction construction is heart-dependent. If \(\A\) and \(\B\)
are distinct finite-length hearts in \(D\), then semisimple objects,
kernels, and semisimple detectability depend on the chosen heart. In
particular, if \(X\) lies in both \(\A\) and \(\B\), the obstruction
subobjects $\mathcal O^{\A}(X)$ and $\mathcal O^{\B}(X)$ need not agree.
More generally, the obstruction theory of objects may change under a
change of heart. Consequently, both semisimple reconstruction and
obstruction depth may vary under a change of heart.
\end{remark}

Thus changing heart may alter both semisimple detectability and the
resulting obstruction theory. To understand this dependence, we now
examine the effect of HRS tilting.

\subsection{HRS Tilting}

Among the most important operations on finite-length hearts is Happel--Reiten--Smal{\o} tilting. From the perspective of obstruction theory, the central question is how semisimple detectability and the associated obstruction subobjects change under a variation of heart. Let \((\T,\F)\) be a torsion pair in a finite-length heart \(\A\), and let $\A^\sharp=\langle \T,\F[1]\rangle$
be the associated Happel--Reiten--Smal{\o} tilt.

\begin{theorem}
\label{thm:tilting_creates_obstruction}
Let $0\to F'\to F\to T\to0$ be a nonsplit exact sequence in \(\A\), with $F,F'\in\F$, $T\in\T$. Then in \(\A^\sharp\) there is a short exact sequence
$0\to T\to F'[1]\to F[1]\to0$. If, moreover, \(T\) is simple in \(\A^\sharp\), \(F[1]\) is semisimple in
\(\A^\sharp\), and the induced sequence is nonsplit, then $\mathcal O^{\A^\sharp}(F'[1]) \cong T$.
\end{theorem}

\begin{proof}
The short exact sequence in \(\A\) $0\to F'\to F\to T\to0$ determines a distinguished triangle $F'\to F\to T\to F'[1]$.
Rotating yields $T\to F'[1]\to F[1]\to T[1]$. Since $T\in\T$, $F,F'\in\F$, the objects $T$, $F'[1]$, $F[1]$ all lie in the tilted heart \(\A^\sharp\). Hence the rotated triangle becomes a short exact sequence in \(\A^\sharp\): $0\to T\to F'[1]\to F[1]\to0$. Let $X=F'[1]$. Since the quotient map $X\to F[1]$ has semisimple target, the definition of the obstruction subobject gives $\mathcal O^{\A^\sharp}(X)\subseteq\ker(X\to F[1])=T$.

Conversely, let $g:X\to S$ be a morphism in \(\A^\sharp\), where \(S\) is semisimple. We claim that \(g\) annihilates \(T\). Suppose that the composite $h:T\to X\to S$ were nonzero. Since \(T\) is simple, \(h\) is a monomorphism, and since \(S\) is semisimple, its image is a direct summand of \(S\). Hence there exists a morphism $r:S\to T$ such that $r\circ h=\mathrm{id}_T$. Writing $i:T\to X$ for the inclusion, we have $h=g\circ i$, and therefore $(r\circ g)\circ i=\mathrm{id}_T$. Thus \(i\) splits, forcing the short exact sequence $0\to T\to X\to F[1]\to0$ to split, contradicting the nonsplitting assumption. Therefore every morphism from \(X\) into a semisimple object annihilates \(T\), so $T\subseteq\mathcal O^{\A^\sharp}(X)$. Combining the two inclusions gives $\mathcal O^{\A^\sharp}(X)=T$ as subobjects of \(X\), and hence $\mathcal O^{\A^\sharp}(F'[1])\cong T$.
\end{proof}

The theorem shows that obstruction can be created by tilting even when
it was absent in the original heart. Extension data that are invisible
from the perspective of semisimple reconstruction in \(\A\) may become
visible as genuine obstruction layers in the tilted heart
\(\A^\sharp\).

\begin{corollary}
\label{cor:tilting_destroys_semisimplicity}
There exist finite-length hearts \(\A\) and HRS tilts \(\A^\sharp\) such
that an object semisimple in \(\A\) acquires nonzero obstruction in
\(\A^\sharp\).
\end{corollary}

\begin{proof}
Let $Q:1\longrightarrow2$ be the \(A_2\)-quiver, and let $\A=\operatorname{Rep}_k(Q)$. Write $S_1=(k\to0)$, $S_2=(0\to k)$, and $P_1=(k\xrightarrow{\mathrm{id}}k)$. There is a nonsplit exact sequence in \(\A\) $0\to S_2\to P_1\to S_1\to0$. Let $\T=\langle S_1\rangle$ and $\F=\{M\in\A:\operatorname{Hom}_{\A}(S_1,M)=0\}$. Then $S_2,P_1\in\F$ and $S_1\in\T$. Let
$\A^\sharp=\langle\T,\F[1]\rangle$ be the corresponding HRS tilt. The given short exact sequence determines a distinguished triangle $S_2\to P_1\to S_1\to S_2[1]$, whose rotation gives
$S_1\to S_2[1]\to P_1[1]\to S_1[1]$. Hence there is a short exact sequence in \(\A^\sharp\) $0\to S_1\to S_2[1]\to P_1[1]\to0$. This sequence is nonsplit. Indeed, if it split in \(\A^\sharp\), then
the associated distinguished triangle would be split. Since splitting
of distinguished triangles is preserved under rotation, the original
triangle $S_2\to P_1\to S_1\to S_2[1]$ would also be split, contradicting
the fact that $0\to S_2\to P_1\to S_1\to0$ is nonsplit in \(\A\).

Moreover, \(S_1\) is simple in \(\A^\sharp\). Indeed, \(\T=\langle S_1\rangle\), so \(S_1\) is simple in the torsion part, and no nonzero object of \(\F[1]\) can occur as a subobject of \(S_1\) in the tilted heart. Also \(P_1[1]\) is simple in \(\A^\sharp\). The indecomposable objects of \(\F\) are \(S_2\) and \(P_1\). Although \(S_2\) is a proper subobject of \(P_1\) in \(\A\), its quotient is \(S_1\notin\F\), so it does not define a proper subobject of \(P_1\) in the exact structure induced on \(\F\). Thus \(P_1\) is simple in \(\F\). Moreover, $\operatorname{Ext}^1_{\A}(S_1,P_1)=0$, so no nonzero object of \(\T=\langle S_1\rangle\) can occur as a subobject of \(P_1[1]\) in \(\A^\sharp\). Hence \(P_1[1]\) is simple in the tilted heart, and therefore semisimple. Thus all the hypotheses of Theorem~\ref{thm:tilting_creates_obstruction}
are satisfied, and therefore
$\mathcal O^{\A^\sharp}(S_2[1])\cong S_1$. Since \(S_1\neq0\), the object \(S_2[1]\) has nonzero obstruction in
\(\A^\sharp\), whereas \(S_2\) is simple, and hence semisimple, in the
original heart \(\A\). This proves the claim.
\end{proof}

The example demonstrates that obstruction created by tilting is not an
isolated phenomenon. Rather, variation of heart can reveal previously
hidden layers of obstruction. This suggests recording systematically
which torsion objects appear as obstruction layers after tilting.

\subsection{Obstruction Shadows}

The obstruction layers produced by HRS tilting naturally organize into a torsion-theoretic invariant. The following construction records the portion of the torsion class that becomes visible through obstruction in the tilted heart.

\begin{definition}
\label{def:torsion_obstruction_shadow}
Let \((\T,\F)\) be a torsion pair in \(\A\), and suppose that \(\T\) is
hereditary. Let \(\A^\sharp=\langle\T,\F[1]\rangle\) be the associated
HRS tilt. The \emph{torsion obstruction shadow} is the Serre subcategory
\(\Shad_{\T}(\A^\sharp/\A)\subseteq\T\) generated by the objects
\(H^0_{\A}\bigl(\mathcal O^{\A^\sharp}(F[1])\bigr)\), \(F\in\F\).
\end{definition}

\begin{definition}
\label{def:visible_torsion}
A simple object \(T\in\T\) is called \emph{visible} if \(T\) occurs as a composition factor of \(H^0_{\A}\bigl(\mathcal O^{\A^\sharp}(F[1])\bigr)\) for some \(F\in\F\). Let \(\T_{\mathrm{vis}}\subseteq\T\) be the Serre subcategory generated by the visible simple objects of \(\T\).
\end{definition}

Theorem~\ref{thm:tilting_creates_obstruction} gives a concrete mechanism for visibility. In particular, if a simple object \(T\in\T\) occurs in a nonsplit exact sequence \(0\to F'\to F\to T\to0\) with \(F,F'\in\F\), such that the induced sequence \(0\to T\to F'[1]\to F[1]\to0\) in \(\A^\sharp\) is nonsplit and \(F[1]\) is semisimple, then \(\mathcal O^{\A^\sharp}(F'[1])\cong T\), and hence \(T\) is visible.

\begin{theorem}
\label{thm:obstruction_shadow_characterization}
With notation as above, \(\Shad_{\T}(\A^\sharp/\A)=\T_{\mathrm{vis}}\).
\end{theorem}

\begin{proof}
By definition, \(\Shad_{\T}(\A^\sharp/\A)\) is the Serre subcategory generated by the objects \(H^0_{\A}\bigl(\mathcal O^{\A^\sharp}(F[1])\bigr)\), with \(F\in\F\). Since \(\A\) has finite length, each such object has finite length and therefore belongs to the Serre subcategory generated by its simple composition factors. By Definition~\ref{def:visible_torsion}, every such composition factor is a visible simple object. Hence every generator of \(\Shad_{\T}(\A^\sharp/\A)\) belongs to \(\T_{\mathrm{vis}}\), and therefore \(\Shad_{\T}(\A^\sharp/\A)\subseteq\T_{\mathrm{vis}}\). Conversely, every visible simple object \(T\) is, by definition, a composition factor of \(H^0_{\A}\bigl(\mathcal O^{\A^\sharp}(F[1])\bigr)\) for some \(F\in\F\). Since \(\Shad_{\T}(\A^\sharp/\A)\) is a Serre subcategory containing this object, it contains all of its composition factors, and hence contains \(T\). Since \(\T_{\mathrm{vis}}\) is the Serre subcategory generated by all visible simple objects, it follows that \(\T_{\mathrm{vis}}\subseteq\Shad_{\T}(\A^\sharp/\A)\). Thus \(\Shad_{\T}(\A^\sharp/\A)=\T_{\mathrm{vis}}\).
\end{proof}

\section{Obstruction Theory Beyond Radical Theory}
\label{sec:classical}

The obstruction theory developed in the preceding sections was
constructed entirely from semisimple detectability and universal
reconstruction. In a fixed finite-length abelian category, the resulting
obstruction subobject turns out to recover the classical radical.
However, the obstruction formalism contains substantially more
structure.

Indeed, obstruction theory naturally produces a universal semisimple
quotient, a semisimplification reflector, a functorial obstruction
filtration, and a heart-dependent theory compatible with HRS tilting.
These constructions continue to exist in settings where no classical
radical theory is available. The purpose of this section is therefore
not to identify obstruction theory with radical theory, but rather to
show that classical radical theory is recovered as a special case of a
broader obstruction-theoretic framework.

We now verify this recovery in the classical module setting. Let \(R\)
be an artinian ring and let $\C=\operatorname{mod}\text{-}R$.

\begin{theorem}
\label{thm:recovery_classical_radical}
For every finite-length \(R\)-module \(M\), \(\mathcal O(M)=\operatorname{rad}(M)\). Consequently, the obstruction filtration agrees with the classical radical filtration.
\end{theorem}

\begin{proof}
By Theorem~\ref{thm:universal_semisimple_quotient}, the quotient \(M/\mathcal O(M)\) is semisimple and the quotient map \(q_{\mathcal O}\colon M\to M/\mathcal O(M)\) is universal among morphisms from \(M\) to semisimple modules. In a finite-length module category, \(M/\operatorname{rad}(M)\) is semisimple, and every morphism from \(M\) to a semisimple module annihilates \(\operatorname{rad}(M)\). Hence the quotient map \(q_{\operatorname{rad}}\colon M\to M/\operatorname{rad}(M)\) has the same universal property. Therefore there is an isomorphism \(\varphi\colon M/\mathcal O(M)\xrightarrow{\sim}M/\operatorname{rad}(M)\) such that \(q_{\operatorname{rad}}=\varphi\circ q_{\mathcal O}\). Since \(\varphi\) is an isomorphism, the two quotient maps have the same kernel, and hence \(\mathcal O(M)=\operatorname{rad}(M)\). Applying the same identification successively to \(\mathcal O^n(M)\) gives \(\mathcal O^{n}(M)=\operatorname{rad}^{n}(M)\) for every \(n\ge0\). Hence the obstruction filtration \(M\supseteq\mathcal O(M)\supseteq\mathcal O^{2}(M)\supseteq\cdots\) coincides term-by-term with the classical radical filtration \(M\supseteq\operatorname{rad}(M)\supseteq\operatorname{rad}^{2}(M)\supseteq\cdots\).
\end{proof}

\begin{corollary}
\label{cor:recovery_jacobson}
For the regular module \(R_R\), $\mathcal O(R_R)=J(R)$.
\end{corollary}

\begin{proof}
Applying Theorem~\ref{thm:recovery_classical_radical} to the regular right \(R\)-module \(R_R\) gives $\mathcal O(R_R) = \operatorname{rad}(R_R)$. For the regular module, the module radical coincides with the Jacobson radical: $\operatorname{rad}(R_R) = J(R)$. Therefore $\mathcal O(R_R) = J(R)$.
\end{proof}

\begin{theorem}
\label{thm:loewy_recovery}
For every finite-length \(R\)-module \(X\), the obstruction filtration agrees
term-by-term with the usual radical filtration. Consequently,
\(\operatorname{depth}_{\mathcal O}(X)=\LoewyLength(X)\).
\end{theorem}

\begin{proof}
By Theorem~\ref{thm:recovery_classical_radical},
\(\mathcal O(Y)=\operatorname{rad}(Y)\) for every finite-length
\(R\)-module \(Y\). Applying this identification inductively gives
\(\mathcal O^{n}(X)=\operatorname{rad}^{n}(X)\) for every \(n\ge 0\).
Hence the obstruction filtration
\(X\supseteq\mathcal O(X)\supseteq\mathcal O^{2}(X)\supseteq\cdots\)
coincides term-by-term with the classical radical filtration
\(X\supseteq\operatorname{rad}(X)\supseteq\operatorname{rad}^{2}(X)
\supseteq\cdots\). Therefore the least integer \(n\) for which
\(\mathcal O^{n}(X)=0\) is the least integer \(n\) for which
\(\operatorname{rad}^{n}(X)=0\). By definition, these are respectively
\(\operatorname{depth}_{\mathcal O}(X)\) and \(\LoewyLength(X)\).
Thus
\(\operatorname{depth}_{\mathcal O}(X)=\LoewyLength(X)\).
\end{proof}

\begin{remark}
The Loewy Recovery Theorem should be viewed as a compatibility result
rather than the primary motivation for the construction. The invariant
\(\mathcal O(X)\) is defined through universal semisimple detectability
and reconstruction, without reference to radicals or Loewy theory.
The theorem shows that the classical radical filtration and Loewy length
emerge naturally from the obstruction-theoretic framework. In this
sense, Loewy theory is recovered as a consequence of semisimple
obstruction rather than taken as a starting point.
\end{remark}

\begin{theorem}
\label{thm:obstruction_beyond_radical}
The obstruction formalism extends the recovery of the classical radical
in finite-length module categories. More precisely:
\begin{enumerate}
\item in every finite-length module category, \(\mathcal O\) recovers the
classical radical filtration;

\item the assignment \(X\mapsto X/\mathcal O(X)\) determines a universal
semisimple reconstruction functor and a reflector onto the semisimple
subcategory;

\item under variation of heart in a triangulated category, the
heart-relative obstruction \(\mathcal O^{\A}\) can change, and under HRS
tilting its obstruction layers can recover torsion-theoretic data.
\end{enumerate}
Consequently, radical theory is recovered as the fixed-heart
module-theoretic instance of a broader obstruction theory.
\end{theorem}

\begin{proof}
Statement \textup{(1)} follows from
Theorems~\ref{thm:recovery_classical_radical}
and~\ref{thm:loewy_recovery}. The former shows that
$\mathcal O(M)=\operatorname{rad}(M)$ for every finite-length
\(R\)-module \(M\), while the latter shows that the obstruction
filtration agrees term-by-term with the classical radical filtration.
Statement \textup{(2)} follows from
Theorems~\ref{thm:universal_semisimple_quotient},
\ref{thm:functoriality}, and~\ref{thm:semisimplification_reflector}.
These show that \(X/\mathcal O(X)\) is semisimple, that every morphism
from \(X\) into a semisimple object factors uniquely through the
canonical quotient \(X\to X/\mathcal O(X)\), and that these universal
quotients are functorial in \(X\). Hence they determine a reflector onto
the semisimple subcategory. Statement \textup{(3)} follows from the
heart-dependent construction and the HRS tilting results. In particular,
Theorem~\ref{thm:tilting_creates_obstruction} shows that tilting can
create new obstruction layers, while
Theorem~\ref{thm:obstruction_shadow_characterization} shows that the
resulting obstruction shadow recovers the visible torsion subcategory.
Thus the equality \(\mathcal O(M)=\operatorname{rad}(M)\) in module
categories is a recovery theorem, not a reduction of the obstruction
formalism to radical theory. The radical is recovered inside a fixed
module category, while the obstruction theory also provides universal
reconstruction, reflection, persistence, and heart-dependent tilting
data.
\end{proof}

\begin{remark}
The preceding results should be interpreted as recovery theorems rather
than identification theorems.

Within a fixed finite-length module category, the obstruction subobject
coincides with the classical radical, the obstruction filtration
coincides with the radical filtration, and obstruction depth coincides
with Loewy length. However, the obstruction-theoretic framework contains
additional structure not present in classical radical theory. In
particular, it provides:
\begin{enumerate}
\item a universal semisimple reconstruction functor;
\item a reflector onto the semisimple subcategory;
\item a persistence theory for obstruction depth;
\item a heart-dependent obstruction theory in triangulated categories;
\item obstruction shadows that recover tilting data.
\end{enumerate}

Thus the radical appears as one manifestation of a more general
obstruction-theoretic phenomenon. The novelty of the theory lies not in
producing a different subobject in module categories, but in providing a
categorical reconstruction framework that extends beyond the classical
scope of radical theory.
\end{remark}

\section{Examples}
\label{sec:examples}

The following examples illustrate the basic regimes of obstruction
theory: vanishing obstruction for semisimple objects, obstruction
arising from nonsplit extensions, heart-dependent obstruction created
by HRS tilting, and the recovery of torsion data through obstruction
shadows.

\subsection{A nonsplit extension}

\begin{example}
Let $0\to S\to E\to T\to0$ be a nonsplit extension of simple objects,
and assume that every morphism from \(E\) to a semisimple object factors
through the quotient \(E\to T\). Then $\mathcal O(E)=S$ as subobjects of
\(E\). Indeed, since \(T\) is simple, it is semisimple, and hence the
quotient map \(E\to T\) occurs among the morphisms defining
\(\mathcal O(E)\). Therefore
\(\mathcal O(E)\subseteq\ker(E\to T)=S\). Conversely, every morphism from \(E\) to a semisimple object factors
through \(E\to T\) by assumption, and therefore annihilates \(S\).
Hence \(S\) is contained in the kernel of every morphism from \(E\) to
a semisimple object, so \(S\subseteq\mathcal O(E)\). Thus
\(\mathcal O(E)=S\). Since \(S\) is simple, it is semisimple, and therefore
\(\mathcal O(S)=0\). Hence \(\mathcal O^2(E)=0\). On the other hand,
\(\mathcal O(E)=S\neq0\), so
\(\operatorname{depth}_{\mathcal O}(E)=2\). Thus the obstruction is
exactly the subobject defining the nonsplit extension.
\end{example}

\subsection{The \(A_2\)-quiver and HRS tilting}

\begin{example}
Let $Q:1\longrightarrow2$ and let
$\A=\operatorname{Rep}_k(Q)$. Write
$S_1=(k\to0)$, $S_2=(0\to k)$, and
$P_1=(k\xrightarrow{\mathrm{id}}k)$. There is a nonsplit exact
sequence $0\to S_2\to P_1\to S_1\to0$. Let
$\T=\langle S_1\rangle$ and
$\F=\{M\in\A:\operatorname{Hom}_{\A}(S_1,M)=0\}$.
Then $S_1\in\T$ and $S_2,P_1\in\F$. Let $\A^\sharp=\langle\T,\F[1]\rangle$ be the associated HRS tilt.
Rotating the triangle associated to the nonsplit sequence yields a
short exact sequence in $\A^\sharp$,
$0\to S_1\to S_2[1]\to P_1[1]\to0$.
In the tilted heart, $S_1$ and $P_1[1]$ are simple.
By Theorem~\ref{thm:tilting_creates_obstruction}, the subobject
$S_1\subseteq S_2[1]$ is contained in
$\mathcal O^{\A^\sharp}(S_2[1])$. Conversely, since $P_1[1]$ is
simple and hence semisimple, the quotient map
$S_2[1]\to P_1[1]$ gives
$\mathcal O^{\A^\sharp}(S_2[1])\subseteq S_1$.
Therefore
$\mathcal O^{\A^\sharp}(S_2[1])\cong S_1$. Since $S_1$ is simple,
$\mathcal O^{\A^\sharp}(S_1)=0$. Thus
$\mathcal O^{\A^\sharp}(S_2[1])\neq0$ while
$\mathcal O^{2,\A^\sharp}(S_2[1])=0$, and hence
$\operatorname{depth}_{\mathcal O}^{\A^\sharp}(S_2[1])=2$. In the original heart $\A$, the object $S_2$ is simple and hence
$\mathcal O^{\A}(S_2)=0$. Thus HRS tilting creates a nontrivial
obstruction layer:
$\mathcal O^{\A}(S_2)=0$, whereas
$\mathcal O^{\A^\sharp}(S_2[1])\cong S_1$.
\end{example}

\subsection{Obstruction shadows}

\begin{example}
In the \(A_2\)-example above, consider the torsion pair
$\T=\langle S_1\rangle$,
$\F=\{M\in\A:\operatorname{Hom}_{\A}(S_1,M)=0\}$.
The obstruction shadow is the Serre subcategory
$\Shad_{\T}(\A^\sharp/\A)\subseteq\T$ generated by the objects
$H^0_{\A}\bigl(\mathcal O^{\A^\sharp}(F[1])\bigr)$ for $F\in\F$.
Since $S_2\in\F$ and, by the preceding example,
$\mathcal O^{\A^\sharp}(S_2[1])\cong S_1$, we have
$H^0_{\A}\bigl(\mathcal O^{\A^\sharp}(S_2[1])\bigr)\cong S_1$.
Therefore $S_1\in\Shad_{\T}(\A^\sharp/\A)$.
Since $\Shad_{\T}(\A^\sharp/\A)$ is a Serre subcategory and
$\T=\langle S_1\rangle$, it follows that
$\T\subseteq\Shad_{\T}(\A^\sharp/\A)$.
The reverse inclusion holds by construction, so
$\Shad_{\T}(\A^\sharp/\A)=\T$.
Thus the obstruction shadow recovers the entire torsion class. This
illustrates concretely how obstruction data arising from tilting can
recover the torsion-theoretic structure used to construct the tilt.
\end{example}

\section{Conclusion}
\label{sec:conclusion}

We have developed a semisimple obstruction theory for finite-length
abelian categories based on the principle that semisimple objects act
as universal detectors of visible structure. The obstruction subobject
\(\mathcal O(X)\) records the part of an object invisible to every
morphism into a semisimple object. It determines a universal
semisimple quotient and a canonical semisimplification reflector, while
its iteration yields an obstruction filtration measuring the persistence
of non-semisimple structure.

In finite-length module categories, this filtration recovers the
classical radical filtration and Loewy theory. These are recovery
results rather than definitions: the obstruction construction is
intrinsic to the ambient abelian category and requires no prior
ring-theoretic or radical-theoretic structure. Classical radical theory
therefore appears as the fixed-heart module-theoretic instance of the
more general semisimple obstruction formalism.

The dependence of obstruction on the ambient heart extends the theory
beyond this classical setting. Under Happel--Reiten--Smal{\o} tilting,
extension data can become new obstruction layers, and the resulting
obstruction shadows can recover torsion-theoretic information associated
with the tilt. Thus semisimple obstruction provides a common categorical
framework for universal reconstruction, persistence, and the variation
of semisimple detectability across hearts.

\bibliographystyle{amsplain}
\bibliography{references}

@book{AndersonFuller1992Rings,
 author = {Anderson, F. W. and Fuller, K. R.},
 title = {Rings and categories of modules.},
 edition = {2. ed.},
 fseries = {Graduate Texts in Mathematics},
 series = {Grad. Texts Math.},
 issn = {0072-5285},
 volume = {13},
 isbn = {0-387-97845-3},
 year = {1992},
 publisher = {New York: Springer-Verlag},
 language = {English},
 zbMATH = {107550},
 Zbl = {0765.16001}
}

@book{AssemSimsonSkowronski,
 author = {Assem, Ibrahim and Simson, Daniel and Skowro{\'n}ski, Andrzej},
 title = {Elements of the representation theory of associative algebras. {Vol}. 1: {Techniques} of representation theory.},
 fseries = {London Mathematical Society Student Texts},
 series = {Lond. Math. Soc. Stud. Texts},
 issn = {0963-1631},
 volume = {65},
 isbn = {0-521-58631-3; 0-521-58423-X},
 year = {2006},
 publisher = {Cambridge: Cambridge University Press},
 language = {English},
 zbMATH = {2228448},
 Zbl = {1092.16001}
}

@book{Benson,
 author = {Benson, D. J.},
 title = {Representations and cohomology. {I}: {Basic} representation theory of finite groups and associative algebras},
 fseries = {Cambridge Studies in Advanced Mathematics},
 series = {Camb. Stud. Adv. Math.},
 volume = {30},
 isbn = {0-521-63653-1},
 year = {1998},
 publisher = {Cambridge: Cambridge University Press},
 language = {English},
 zbMATH = {1181204},
 Zbl = {0908.20001}
}

@misc{CurtisReiner,
 author = {Curtis, Charles W. and Reiner, Irving},
 title = {Methods of representation theory, with applications to finite groups and orders. {Vol}. {I}},
 year = {1981},
 language = {English},
 howpublished = {Pure and {Applied} {Mathematics}. {A} {Wiley}-{Interscience} {Publication}. {New} {York} etc.: {John} {Wiley} \& {Sons}. {XXI}, 819 p. {{\textsterling}} 40.70 (1981).},
 zbMATH = {3736029},
 Zbl = {0469.20001}
}

@book{EtingofGelaki2015Tensor,
 author = {Etingof, P. and Gelaki, S. and Nikshych, D. and Ostrik, V.},
 title = {Tensor categories},
 fseries = {Mathematical Surveys and Monographs},
 series = {Math. Surv. Monogr.},
 issn = {0076-5376},
 volume = {205},
 isbn = {978-1-4704-2024-6},
 year = {2015},
 publisher = {Providence, RI: American Mathematical Society (AMS)},
 language = {English},
 doi = {10.1090/surv/205},
 zbMATH = {6444410},
 Zbl = {1365.18001}
}

@book{Faith1973Algebra,
 author = {Faith, C.},
 title = {Algebra: {Rings}, modules and categories. {I}},
 fseries = {Grundlehren der Mathematischen Wissenschaften},
 series = {Grundlehren Math. Wiss.},
 issn = {0072-7830},
 volume = {190},
 year = {1973},
 publisher = {Springer, Cham},
 language = {English},
 zbMATH = {3419292},
 Zbl = {0266.16001}
}

@article{Freyd,
 author = {Freyd, Peter},
 title = {Abelian categories},
 fjournal = {Reprints in Theory and Applications of Categories},
 journal = {Repr. Theory Appl. Categ.},
 volume = {2003},
 number = {3},
 pages = {xxiii + 1--164},
 year = {2003},
 language = {English},
 zbMATH = {2028925},
 Zbl = {1041.18001}
}

@book{HappelTriangulated,
 author = {Happel, Dieter},
 title = {Triangulated categories in the representation theory of finite dimensional algebras},
 fseries = {London Mathematical Society Lecture Note Series},
 series = {Lond. Math. Soc. Lect. Note Ser.},
 issn = {0076-0552},
 volume = {119},
 isbn = {0-521-33922-7},
 year = {1988},
 publisher = {Cambridge (UK) etc.: Cambridge University Press},
 language = {English},
 zbMATH = {40982},
 Zbl = {0635.16017}
}

@book{HRS,
 author = {Happel, Dieter and Reiten, Idun and Smal{\o}, Sverre O.},
 title = {Tilting in abelian categories and quasitilted algebras},
 fseries = {Memoirs of the American Mathematical Society},
 series = {Mem. Am. Math. Soc.},
 issn = {0065-9266},
 volume = {575},
 isbn = {978-0-8218-0444-5; 978-1-4704-0160-3},
 year = {1996},
 publisher = {Providence, RI: American Mathematical Society (AMS)},
 language = {English},
 doi = {10.1090/memo/0575},
 zbMATH = {867897},
 Zbl = {0849.16011}
}

@book{Jacobson1956Structure,
  author    = {Jacobson, N.},
  title     = {Structure of Rings},
  series    = {American Mathematical Society Colloquium Publications},
  volume    = {37},
  publisher = {American Mathematical Society},
  address   = {Providence, RI},
  year      = {1956},
  isbn      = {978-0-8218-1037-8},
  mrnumber  = {0081264},
  zbmath    = {0073.02002}
}

@book{Lam2001FirstCourse,
 author = {Lam, T. Y.},
 title = {A first course in noncommutative rings.},
 edition = {2nd ed.},
 fseries = {Graduate Texts in Mathematics},
 series = {Grad. Texts Math.},
 issn = {0072-5285},
 volume = {131},
 isbn = {0-387-95325-6; 0-387-95183-0},
 year = {2001},
 publisher = {New York, NY: Springer},
 language = {English},
 zbMATH = {1633834},
 Zbl = {0980.16001}
}

@book{MacLane,
 author = {Mac Lane, S.},
 title = {Categories for the working mathematician.},
 edition = {2nd ed},
 fseries = {Graduate Texts in Mathematics},
 series = {Grad. Texts Math.},
 issn = {0072-5285},
 volume = {5},
 isbn = {0-387-98403-8},
 year = {1998},
 publisher = {New York, NY: Springer},
 language = {English},
 doi = {book/10.1007/978-1-4757-4721-8},
 zbMATH = {1216133},
 Zbl = {0906.18001}
}

@book{Neeman,
 author = {Neeman, Amnon},
 title = {Triangulated categories},
 fseries = {Annals of Mathematics Studies},
 series = {Ann. Math. Stud.},
 volume = {148},
 isbn = {0-691-08685-0; 0-691-08686-9},
 year = {2001},
 publisher = {Princeton, NJ: Princeton University Press},
 language = {English},
 doi = {10.1515/9781400837212},
 zbMATH = {1573275},
 Zbl = {0974.18008}
}

@book{Pierce,
 author = {Pierce, R. S.},
 title = {Associative algebras},
 fseries = {Graduate Texts in Mathematics},
 series = {Grad. Texts Math.},
 issn = {0072-5285},
 volume = {88},
 year = {1982},
 publisher = {Springer, Cham},
 language = {English},
 zbMATH = {3783206},
 Zbl = {0497.16001}
}

@book{Serre,
 author = {Serre, Jean-Pierre},
 title = {Linear representations of finite groups. {Translated} from the {French} by {Leonard} {L}. {Scott}},
 fseries = {Graduate Texts in Mathematics},
 series = {Grad. Texts Math.},
 issn = {0072-5285},
 volume = {42},
 year = {1977},
 publisher = {Springer, Cham},
 language = {English},
 zbMATH = {3552764},
 Zbl = {0355.20006}
}

\end{document}